\title{\vspace{-0.7cm}Counting and packing Hamilton $\ell$-cycles in dense hypergraphs}
\author{ Asaf Ferber
\thanks{Institute of Theoretical Computer Science
ETH, 8092 Z\"urich, Switzerland. Email: asaf.ferber@inf.ethz.ch.}
\and Michael Krivelevich
\thanks{School of Mathematical Sciences, Raymond and Beverly Sackler
Faculty of Exact Sciences, Tel Aviv University, Tel Aviv, 69978,
Israel. Email: krivelev@post.tau.ac.il. Research supported in part
by USA-Israel BSF Grant 2010115 and by grant 912/12 from the Israel
Science Foundation.} \and Benny Sudakov
\thanks{Department of Mathematics, ETH, 8092 Z\"urich, Switzerland. Email:
benjamin.sudakov@math.ethz.ch. Research supported in part by SNSF
grant 200021-149111.}}

\documentclass[11pt]{article}
\usepackage{amsmath,amssymb,latexsym,color,epsfig,enumerate,a4}

\newif\ifnotesw\noteswtrue

\date{}

\newtheorem{theorem}{Theorem}[section]
\newtheorem{lemma}[theorem]{Lemma}
\newtheorem{claim}[theorem]{Claim}
\newtheorem{proposition}[theorem]{Proposition}

\newtheorem{remark}[theorem]{Remark}

\newcommand{\Bin}{\ensuremath{\textrm{Bin}}}

\newenvironment{proof}{\noindent{\bf Proof\,}}{\hfill$\Box$}

\begin{document}
\maketitle

\begin{abstract}
A $k$-uniform hypergraph $\mathcal H$ contains a
Hamilton $\ell$-cycle, if there is a cyclic ordering of
the vertices of $\mathcal H$ such that the edges of the cycle are segments of length
$k$ in this ordering and any two consecutive edges $f_i,f_{i+1}$ share exactly
$\ell$ vertices. We consider problems about packing and counting
Hamilton $\ell$-cycles in hypergraphs of large minimum degree. Given
a hypergraph $\mathcal H$, for a $d$-subset $A\subseteq V(\mathcal
H)$, we denote by $d_{\mathcal H}(A)$ the number of distinct
\emph{edges} $f\in E(\mathcal H)$ for which $A\subseteq f$, and set
$\delta_d(\mathcal H)$ to be the minimum $d_{\mathcal H}(A)$ over all
$A\subseteq V(\mathcal H)$ of size $d$. We show that if a
$k$-uniform hypergraph on $n$ vertices $\mathcal H$ satisfies
$\delta_{k-1}(\mathcal H)\geq \alpha n$ for some $\alpha>1/2$, then
for every $\ell<k/2$ $\mathcal H$ contains
$(1-o(1))^n\cdot n!\cdot
\left(\frac{\alpha}{\ell!(k-2\ell)!}\right)^{\frac{n}{k-\ell}}$
Hamilton $\ell$-cycles. The exponent above is easily seen to be
optimal. In addition, we show that if $\delta_{k-1}(\mathcal H)\geq
\alpha n$ for $\alpha>1/2$, then $\mathcal H$ contains $f(\alpha)n$
edge-disjoint Hamilton $\ell$-cycles for an explicit function
$f(\alpha)>0$. For the case where every $(k-1)$-tuple $X\subset V({\mathcal H})$ satisfies
$d_{\mathcal H}(X)\in (\alpha\pm o(1))n$, we show that $\mathcal H$
contains edge-disjoint Haimlton $\ell$-cycles which cover all but
$o\left(|E(\mathcal H)|\right)$ edges of $\mathcal H$.  As a tool we prove the
following result which might be of independent interest: For a
bipartite graph $G$ with both parts of size $n$, with minimum degree
at least $\delta n$, where $\delta>1/2$, and for $p=\omega(\log
n/n)$ the following holds. If $G$ contains an $r$-factor for
$r=\Theta(n)$, then by retaining edges of $G$ with probability $p$
independently at random, w.h.p the resulting graph contains a
$(1-o(1))rp$-factor.

\end{abstract}

\section{Introduction}

Hamiltonicity is definitely one of the most studied properties of
graphs in the last few decades, and many deep and interesting
results have been obtained about it. In his seminal paper \cite{Dirac}, Dirac
proved that every graph on $n$ vertices, $n\ge 3$, with minimum
degree at least $n/2$ is Hamiltonian. The complete bipartite graph
$K_{m,m+1}$ shows that this theorem is best possible, i.e., the
minimum degree condition cannot be improved. Moreover, this extremal
example hints that the addition of one more edge creates many
Hamilton cycles. It thus natural to ask the following questions:

\begin{enumerate} [$(1)$]
\item How many \emph{edge-disjoint} Hamilton cycles does a {\em Dirac graph} (that is, a
graph $G$ on $n$ vertices with minimum degree $\delta(G)\geq n/2$)
have?
\item How many \emph{distinct} Hamilton cycles does a Dirac graph
have?
\end{enumerate}

These questions have been examined by various researchers and many
results are known. Among them are Christofides, K\"uhn and Osthus
\cite{CKO}, Cuckler and Kahn \cite{CK}, K\"uhn, Lapinskas and Osthus
\cite{KLO}, Nash-Williams \cite{NashWilliams1, NashWilliams2,
NashWilliams3}, S\'ark\"ozy, Selkow and Szemer\'edi \cite{SSS}, and
the authors of this paper \cite{FKS}. In particular, it is worth
mentioning very recent remarkable results due to Csaba, K\"uhn, Lo,
Osthus and Treglown \cite{CKLOT} who settled the long standing
conjectures made by Nash-Williams \cite{NashWilliams1,
NashWilliams2, NashWilliams3}. They showed that every $d$-regular
Dirac graph contains $\lfloor \frac{d}{2}\rfloor$ edge-disjoint
Hamilton cycles, and that every graph $G$ on $n$ vertices with
minimum degree $\delta\geq n/2$ contains at least
$reg_{even}(n,\delta)/2$ edge-disjoint Hamilton cycles, where
$reg_{even}(n,\delta)$ denotes the largest degree of an even-regular
spanning subgraph one can guarantee in a graph on $n$ vertices with
minimum degree $\delta$. These results are clearly optimal.

Note that if a graph $G$ contains $r$ edge-disjoint Hamilton cycles,
then in particular $G$ contains a $2r$-\emph{factor}, that is, a
spanning $2r$-regular subgraph. Therefore, the following question is
also related to the two mentioned above:
\begin{enumerate}[$(3)$]
\item Given a graph $G$ with minimum degree $\delta(G)$, what is the
maximal $r$ for which $G$ contains an $r$-factor?
\end{enumerate}

As in Dirac's Theorem, the complete bipartite graph $K_{m,m+1}$ with
unbalanced parts demonstrates that for $\delta(G)<n/2$ one can not
expect to obtain even a $1$-factor. The question about finding the
maximal $r:=r(\delta,n)$ such that any graph $G$ on $n$ vertices
with minimum degree $\delta$ must contain an $r$-factor has also
been investigated by various researchers. Among them are Katerinis
\cite{Katerinis} and Hartke, Martin and Seacrest \cite{HMS}. The
former showed that any Dirac graph contains an $r$-factor for $r\geq
\frac{n+5}{4}$ (he also gave an example of a Dirac graph $G$ on $n$
vertices that does not contain an $\frac{n+6}{4}$-factor), and the
latter generalized the result to graphs with minimum degree
$\delta$, with $\delta\geq n/2$.

In this paper we investigate analogous questions in hypergraphs.
First we need to define the notion of a Hamilton cycle in a
hypergraph. For two positive integers $0\leq \ell < k$, a
$(k,\ell)$-\emph{cycle} is a $k$-uniform hypergraph whose vertices
can be ordered cyclically such that the edges are segments of that
order and such that every two consecutive edges share exactly $\ell$
vertices. In case that $0\leq \ell\leq k/2$, we refer to
$(k,\ell)$-cycles as \emph{loose} cycles. Now, let $\mathcal H$ be a
$k$-uniform hypergraph and let $0\leq \ell<k$. We say that $\mathcal
H$ contains a \emph{Hamilton $\ell$-cycle} if $\mathcal H$ contains
a $(k,\ell)$-cycle using all the vertices of $\mathcal H$. Note
that in the case $\ell=0$ a Hamilton $\ell$-cycle corresponds to a
\emph{perfect matching}.

Analogously to graphs, the connection between the degrees in
hypergraphs and the appearance of Hamilton $\ell$-cycles is well
studied, and many results have been derived. Of course, an obvious
necessary condition for a $k$-uniform hypergraph on $n$ vertices to
contain a Hamilton $\ell$-cycle is for $(k-\ell)$ to divide $n$.
Before we proceed to describe the previous work and to state our
results, let us introduce some notation. Given a hypergraph
$\mathcal H$, for a $d$-subset $A\in \binom{V(\mathcal H)}{d}$, we
denote by $d_{\mathcal H}(A)$ the number of distinct \emph{edges}
$f\in E(\mathcal H)$ for which $A\subseteq f$, and set
$$\delta_d(\mathcal H)=\min d_{\mathcal H}(A), \text{ and }\Delta_d(\mathcal
H)=\max d_{\mathcal H}(A),$$ where the minumum and the maximum are
taken over all subsets $A\subseteq V(\mathcal H)$ of size exactly
$d$.
In a similar way, for two subsets $X,Y\subseteq V(\mathcal H)$, we
denote by $d_{\mathcal H}(X,Y)$ the size of the \emph{neighborhood
of $X$ in $Y$}. That is, $d_{\mathcal H}(X,Y):=|\left\{Z\subseteq Y:
X\cup Z \in E(\mathcal H)\right\}|$. For a fixed set $Y$ and an
integer $d<k$, we set
$$\delta_{d}(Y)=\min d_{\mathcal H}(X,Y), \text{ and }\Delta_{d}(Y)=\max d_{\mathcal
H}(X,Y),$$ where the minimum and maximum are taken over all subsets
$X\subseteq V(\mathcal H)$ of size $d$.


Katona and Kierstead were the first to obtain a Dirac-type result
for hypergraphs. They proved in \cite{KK} that if
$\delta_{k-1}(\mathcal H)\geq \left(1-\frac{1}{2k}\right)n+O_k(1)$,
then $\mathcal H$ contains a Hamilton $(k-1)$-cycle. They also gave
an example for a hypergraph $\mathcal H$ with $\delta_{k-1}(\mathcal
H)=\lfloor \frac{n-k+3}{2}\rfloor$ which does not contain a Hamilton
$(k-1)$-cycle, and implicitly conjectured that this is the correct
bound. For $k=3$, this conjecture has been confirmed by R\"odl,
Ruci\'nski and Szemer\'edi in \cite{RRS}. For $k\geq 4$, it is
proved in \cite{RRS1} that $\delta_{k-1}(\mathcal H)\approx \frac
n2$ is asymptotically the correct bound for the existence of a
Hamilton $(k-1)$-cycle in $\mathcal H$. A construction of
Markstr\"om and Ruci\'nski from \cite{MR} demonstrates that
$\delta_{k-1}(\mathcal H)\approx \frac n2$ is necessary for having a
perfect matching in $\mathcal H$, and since whenever $(k-\ell)|k$, a
Hamilton $\ell$-cycle contains a perfect matching, one obtains that
indeed $\delta_{k-1}(\mathcal H)\approx \frac n2$ is the correct
(asymptotic) bound for enforcing the existence of a Hamilton
$\ell$-cycle for each such $\ell$. For values of $\ell$ for which
$(k-\ell)\nmid k$, K\"uhn, Mycroft and Osthus showed in \cite{KMO}
that $\delta_{k-1}(\mathcal H)\approx \frac{n}{\lceil
\frac{k}{k-\ell}\rceil(k-\ell)}$ is the correct asymptotic bound for
enforcing the existence of a Hamilton $\ell$-cycle. There are many
other important and interesting results regarding the connection
between the minimum degree of a hypergraph and the existence of
Hamilton $\ell$-cycles which we did not mention, and for a more
complete list we refer the reader to the excellent survey of R\"odl
and Ruci\'nski \cite{RR}.

Now we are ready to state our main results. As far as we know, this
paper is the first attempt to deal with Questions (1)--(3) in the
hypergraph setting. In our first theorem we show that a dense
$k$-uniform hypergraph contains the ``correct" number of loose
Hamilton cycles. That is, we show that given a $k$-uniform
hypergraph $\mathcal H$ on $n$ vertices with $\delta_{k-1}(\mathcal
H)\geq \alpha n$, the number of Hamilton $\ell$-cycles in $\mathcal
H$ is at least (up to a sub-exponential factor) the expected number
of Hamilton $\ell$-cycles in a random $k$-uniform hypergraph with
edge probability $p=\alpha$ (that is, a hypergraph obtained by
choosing every $k$-subset of $[n]$ with probability $p$,
independently at random). The expected number of such cycles is
$$(n-1)!\cdot\frac{k-\ell}{2}\cdot\left(\frac{\alpha}{\ell!(k-2\ell)!}\right)^{\frac{n}{k-\ell}}.$$
Indeed, first enumerate the vertices and define the edges of the
$(k,\ell)$-cycle accordingly. Then, in each of the
$\frac{n}{k-\ell}$ edges, divide by the number of ways to order the
first $\ell$ vertices and the next $k-2\ell$ vertices. Finally,
divide by $\frac{2n}{k-\ell}$, which is the number of different ways
to obtain the same cycle.

\begin{theorem} \label{main1}
Let $\ell$ and $k$ be integers satisfying $0\leq \ell< k/2$, and let
$1/2<\alpha\leq 1$. Then, for sufficiently large integer $n$ the
following holds. Suppose that
\begin{enumerate}[$(i)$]
\item $(k-\ell)|n$, and
\item $\mathcal H$ is a $k$-uniform hypergraph on $n$ vertices, and
\item $\delta_{k-1}(\mathcal H)\geq \alpha n$.
\end{enumerate}

Then, the number of Hamilton $\ell$-cycles in $\mathcal H$ is at
least
$$(1-o(1))^n\cdot n!\cdot
\left(\frac{\alpha}{\ell!(k-2\ell)!}\right)^{\frac{n}{k-\ell}}.$$
\end{theorem}

This is an extension to hypergraphs of the result obtained by
Cuckler and Kahn \cite{CK} for the case of graphs. We remark that their bound is
more accurate and is phrased in terms of certain entropy function
over edge weighting of the graph. We will use their result in our
proof.

Since in a $k$-uniform hypergraph $\mathcal H$ on $n$ vertices, a
Hamilton $\ell$-cycle contains $\frac{n}{k-\ell}$ edges, one cannot
hope to find more than $|E(\mathcal H)|/\frac{n}{k-\ell}$
edge-disjoint such cycles. In the following theorem we show that
indeed, up to a multiplicative factor, any dense $k$-uniform
hypergraph $\mathcal H$ contains the correct number of edge-disjoint
loose Hamilton cycles.

\begin{theorem} \label{main2}

Let $k$ and $\ell$ be integers satisfying $0\leq\ell < k/2$, and let
$1/2<\alpha'<\alpha\leq 1$. Then for all sufficiently large $n$ the
following holds. Suppose that
\begin{enumerate}[$(i)$]
\item $(k-\ell)|n$, and
\item $\mathcal H$ is a $k$-uniform hypergraph on $n$ vertices, and
\item $\delta_{k-1}(\mathcal H)\geq \alpha n$.
\end{enumerate}
Then ${\cal H}$ contains at least
$$
(1-o(1))\cdot \frac{ f(\alpha')|E({\cal H})|}{\frac{n}{k-\ell}}
$$
edge-disjoint Hamilton $\ell$-cycles, where
$f(x)=\frac{x+\sqrt{2x-1}}{2}$.

\end{theorem}

We remark that we prove Theorem \ref{main2} by translating the
problem into a problem of graphs. One of the ingredients of our
proof is the ability to find a spanning and regular subgraph of a
dense bipartite graph. In order to achieve this goal we use a result
of Csaba \cite{Csaba} (which is tight for bipartite graphs), and
this is where the function $f$ in Theorem \ref{main2} comes from.

In the special case where the difference between
$\Delta_{k-1}(\mathcal H)$ and $\delta_{k-1}(\mathcal H)$ is small,
we obtain the following asymptotically optimal result.

\begin{theorem}\label{main3}

Let $k$ and $\ell$ be integers  satisfying $0\leq\ell< k/2$, and let
$1/2<\alpha\leq 1$ be a constant. For every $\delta>0$ there exists
$\varepsilon>0$ such that the following holds. For all sufficiently
large $n$, if:
\begin{enumerate}[$(i)$]
\item $(k-\ell)|n$, and
\item $\mathcal H$ is a $k$-uniform hypergraph on $n$ vertices, and
\item $\delta_{k-1}(\mathcal
H)\geq \alpha n$, and
\item $\Delta_{k-1}({\cal H})\le (\alpha+\varepsilon)n$.
\end{enumerate}
Then all but at most $\delta\binom{n}{k}$ edges of $\cal{H}$ can be
packed into Hamilton $\ell$-cycles.
\end{theorem}

Note that Theorem \ref{main3} is more general than the main result
of \cite{FK} in the sense that we do not require any
``pseudo-random" properties of the hypergraph (except, of course,
the assumption that the degrees are large). On the other hand,
Theorem \ref{main3} works only for hypergraphs which are very dense,
but it is known (see e.g. \cite{KMO}) that below the densities we
consider, there are constructions of hypergraphs without Hamilton
cycles.

In the proofs of Theorems \ref{main2} and \ref{main3} we use (as a
tool) the following theorem which is also of independent interest
and is related to the concept of robustness of graph properties (see for
example \cite{KLS}). Before discussing and stating the theorem, let
us introduce the following notation. Let $G$ be a graph. Given a
positive constant $0<p\leq 1$, we say that a graph $G'$ is
distributed according to $G_p$, or $G'\sim G_p$ for brevity, if $G'$
is a subgraph of $G$ obtained by retaining every edge of $G$ with
probability $p$, independently at random. In the following theorem
we show that, given a bipartite graph $G$ with both parts of size $n$
and with $\delta(G)\geq\alpha n$, where $\alpha>1/2$, if $G$
contains an $r$-factor for $r=\Theta(n)$, then for
$p=\omega\left(\frac{\log n}{n}\right)$, a random subgraph $G'\sim
G_p$ typically contains a $(1-o(1))rp$-factor. The proof of the
 theorem appears in Section \ref{subsec:randomfactor}.
%
\begin{theorem}\label{RandomKaterinis}
Let $1/2<\alpha\leq 1$, $\varepsilon>0$ and $0<\rho\leq \alpha$ be
positive constants. Then for sufficiently large integer $n$, the
following holds. Suppose that:
\begin{enumerate}[$(i)$]
\item $G$ is a bipartite graph with parts $A$ and $B$, both of size $n$,
and
\item $\delta(G)\geq \alpha n$, and
\item $G$ contains a $\rho n$-factor.
\end{enumerate}
Then, for $p= \omega\left(\frac{\ln n}{n}\right)$, with probability
$1-n^{-\omega(1)}$ a graph $G'\sim G_p$ has a $k$-factor for
$k=(1-\varepsilon)\rho np$.
\end{theorem}

\begin{remark}\label{remark:notsameprobability}
We remark that the proof of Theorem \ref{RandomKaterinis} is still
valid even if we choose each edge $e\in E(G)$ with probability
$p_e\geq p$. This follows from the monotonicity of the random model $G_p$.
\end{remark}

Let $\mathcal H$ be a $k$-uniform hypergraph on $n$ vertices with
$\delta_{k-1}(\mathcal H)\geq \alpha n$ for some $\alpha>1/2$.
Assume further that $k\mid n$. Now, by applying Theorem \ref{main2}
with $\ell=0$ to $\mathcal H$ one can obtain that $\mathcal H$
contains an $r$-factor for every $r\leq (1-o(1))\frac{f(\alpha)
|E(H)|}{\frac{n}{k}}$. In the following proposition, by slightly
extending a known construction, we show that there are hypergraphs
$\mathcal H$ with $\delta_{k-1}(\mathcal H)\geq n/2-O(1)$ which do
not contain $r$-factors for many values of $r$.

\begin{proposition} \label{main4}
Let $k\leq n$ be positive integers. Then there exists a $k$-uniform
hypergraph $\mathcal H$ on $n$ vertices with $\delta_{k-1}(\mathcal
H)\geq n/2-k-1$, which does not contain an $r$-factor for any odd
integer $r$.
\end{proposition}

\section{Tools}

In this section we introduce the main tools to be used in the proofs
of our results.

%
%

\subsection{Probabilistic tools}

We need to employ standard bounds on large deviations of random
variables. We mostly use the following well-known bound on the lower
and the upper tails of the Binomial distribution due to Chernoff
(see \cite{AS}, \cite{JLR}).

\begin{lemma}\label{Che}
Let $X \sim \emph{\text{Bin}}(n,p)$ and let $\mu=\mathbb{E}(X)$.
Then
\begin{itemize}
    \item $\Pr\left[X<(1-a)\mu\right]<e^{-a^2\mu/2}$ for every
    $a>0$;
    \item $\Pr\left[X>(1+a)\mu\right]<e^{-a^2\mu/3}$ for every $0<a<3/2.$
\end{itemize}
\end{lemma}

\emph{Remark: The conclusions of Lemma \ref{Che} remain the same
when $X$ has the hypergeometric distribution (see \cite{JLR},
Theorem 2.10).}

\noindent The following is a trivial yet useful bound.
\begin{lemma}\label{Che2}
Let $X \sim \emph{\Bin}(n,p)$ and $k \in \mathbb{N}$.Then the
following holds: $$\Pr(X\geq k) \leq \left(\frac{enp}{k}\right)^k.$$
\end{lemma}
\begin{proof}
$\Pr(X \geq k) \leq \binom{n}{k}p^k \leq
\left(\frac{enp}{k}\right)^k$.
\end{proof}

\bigskip We also make an extensive use of the following inequality,
whose proof can be found at \cite{Mac}, Section 3.2.

\begin{theorem} \label{talagrand}
Let $S_n$ denote the set of permutations of $[n]$ and let
$f:S_n\rightarrow \mathbb{R}$ be such that $|f(\pi)-f(\pi')|\leq u$
whenever $\pi'$ is obtained from $\pi$ by transposing two elements.
Then if $\pi$ is chosen randomly from $S_n$ then

$$\Pr\left[|f(\pi)-\mathbb{E}(f)|\geq t\right]\leq
2\exp\left(-\frac{2t^2}{nu^2}\right).$$
\end{theorem}

\subsection{Factors in graphs}

In the proofs of our main results we translate the problem from
hypergraphs to graphs by introducing some auxiliary graphs and then
by trying to find large factors in each such graph. For this goal we
will make use of the following theorem due to Csaba \cite{Csaba}.

\begin{theorem}\label{Csaba}
Let $G=(A\cup B,E)$ be a bipartite graph with parts of size $n$ and
with minimum degree $\delta(G)\geq n/2$. Then $G$ contains a
$\lfloor \rho n \rfloor $-factor for
$\rho=\frac{\delta+\sqrt{2\delta-1}}{2}$, where
$\delta:=\delta(G)/n$.
\end{theorem}

In case the graph is almost regular, a better bound can be obtained
as stated in the following theorem:

\begin{theorem}\label{FactorInAlmostRegular}
For every $\alpha>1/2$ there exist $\varepsilon_0>0$ and an integer
$n_0=n(\alpha)$ such that for every $n\geq n_0$ and $\varepsilon\leq
\varepsilon_0$ the following holds. Suppose that:
\begin{enumerate}[$(i)$]
\item $G$ is a bipartite graph with two parts $A$ and $B$ of size
$n$, and
\item $d_G(v)\in (\alpha\pm \varepsilon)n$ for every $v\in V(G)$.
\end{enumerate}
Then, for every $r\leq (\alpha-10\sqrt{\varepsilon})n$, $G$ contains
an $r$-factor.
\end{theorem}

\begin{proof}[Sketch] Before we sketch the proof, note that there exists a standard bijection between bipartite graphs with parts of size $n$ and digraphs (self loops are allowed!) on $n$ vertices.
For showing it, assume that $G=(A\cup B,E)$ is a bipartite graph
with $|A|=|B|=n$, and define a digraph $D=(A,E')$ as follows (we may
assume that $A=B$): the arc $ab\in E'$ if and only if the
corresponding edge appears in $G$. For the other direction, assume
that $D=(V,E)$ is a digraph. Define $G=(A\cup B,E')$ as follows: the
parts $A$ and $B$ are two copies of $V$. An edge $ab\in E'$ if and
only if the arc $ab\in E$. Now, note that by deleting at most one
edge adjacent to each vertex one can delete all loops and the proof
of Theorem \ref{FactorInAlmostRegular} follows immediately by
combining Lemmas 13.2 and 5.2 of \cite{KO}.
\end{proof}

In addition, we make use of the following theorem due to Cuckler and
Kahn, which provides a good lower bound on the number of perfect
matchings in a bipartite graph with respect to the minimum degree
(see \cite{CK}, p.3).

\begin{theorem}\label{CK}
Let $G$ be a bipartite graph with both parts of size $m$, and let
$\delta(G)=\delta m\geq m/2$ be its minimum degree. Then the number
of perfect matchings in $G$ is at least
$$\delta^m\cdot m!\left(1-o(1)\right)^m.$$
\end{theorem}

%

\subsection{Factors in random subgraphs of dense
graphs}\label{subsec:randomfactor}

In this subsection we prove Theorem  \ref{RandomKaterinis}. In the
proof we make use of the following condition for having a $k$-factor
in a bipartite graph due to Gale and Ryser \cite{GaleRyser} (a proof
can also be found at \cite{Lovasz}, Problem 7.16).

\begin{proposition} \label{Pr1} A bipartite graph
$G=(A\cup B,E)$ with $|A|=|B|$ contains an $r$-factor if and only if
for all $X\subseteq A$ and $Y\subseteq B$ the following holds:
$$r|X|\leq
e_G(X,Y)+r(|B|-|Y|). $$
\end{proposition}

Now we are ready to prove Theorem \ref{RandomKaterinis}.

\begin{proof} Let $G$ be a graph as described in the theorem. We wish to show that a graph $G'\sim G_p$ is w.h.p such that
$$(*) \text{  }k|X|\leq
e_{G_p}(X,Y)+k(n-|Y|),$$ for all $X\subseteq A$ and $Y\subseteq B$,
where $k=(1-\varepsilon)\rho np$ (and then by Proposition \ref{Pr1}
we are done). We distinguish between several cases and consider each
of them separately:

{\bf Case 1:} $|X|+|Y|\leq n$. In this case, since $n-|Y|\geq |X|$,
it follows that $(*)$ is trivial.

{\bf Case 2:}  $|X|+|Y|>n$ (that is, $|Y|\geq n-|X|+1$) and $|X|\leq
f(n)$, where $f(n)=n/\ln\ln n$. Here, since $|Y|>n-|X|=(1-o(1))n$,
$\delta(G)=\alpha n$ and $\alpha>1/2$, it follows that $e_G(X,Y)\geq
(1-o(1))\alpha n|X|$. Using the fact that $e_{G_p}(X,Y)$ is
binomially distributed, applying Chernoff and the union bound we
obtain that
\begin{align*}
\Pr(\exists \text{ such } X,Y \text{ with }& e_{G_p}(X,Y)\leq
(1-\varepsilon/2)e_G(X,Y)p)\leq \sum_{x=1}^{f(n)}\sum_{y=n-x+1}^n
\binom{n}{x}\binom{n}{y}e^{-\Theta(npx)}\\
&=\sum_{x=1}^{f(n)}
\binom{n}{x}\left(\sum_{y=n-x+1}^n\binom{n}{y}\right)e^{-\Theta(npx)}\\
&\leq \sum_{x=1}^{f(n)}x\binom{n}{x}\binom{n}{x-1}
e^{-\Theta(npx)}\\
&=\sum_{x=1}^{f(n)}\frac{x^2}{n-x+1}\binom{n}{x}^2
e^{-\Theta(npx)}\\
&\leq \left(f(n)\right)^2\cdot\sum_{x=1}^{f(n)}\binom{n}{x}^2
e^{-\Theta(npx)},
\end{align*}

which is (recall that $np=\omega(\ln n)$) at most
$$\left(f(n)\right)^2\cdot\sum_{x=1}^{f(n)}\left(\frac{e^2n^2}{x^2}e^{-\omega(\ln n)}\right)^x=n^{-\omega(1)}.$$

Hence, since $(1-\varepsilon)\rho<(1-2\varepsilon/3)\alpha$, it
follows that with probability $1-n^{-\omega(1)}$ we have
$$e_{G_p}(X,Y)+(1-\varepsilon)\rho np(n-|Y|)\geq e_{G_p}(X,Y)$$ $$\geq (1-2\varepsilon/3)\alpha np|X|\geq (1-\varepsilon)\rho np|X|,$$ for each such $X$
and $Y$, and $(*)$ holds.

{\bf Case 3:} $|X|+|Y|>n$ and $|X|>f(n)$. Let

$$\eta_G(x,y)=\min\{e_G(X,Y):X\subseteq A,\text{ } Y\subseteq
B,\text{ }|X|=x, \text{ and } |Y|=y\}.$$

Note that by our assumptions we only consider $x$ and $y$ for which
$x\geq f(n)$ and $x+y\geq n+1$.

Clearly, $e_G(X,Y)\geq x(\alpha n+y-n)$ and $e_G(X,Y)\geq y(\alpha
n+x-n)$ for arbitrary sets $X\subseteq A$ and $Y\subseteq B$ of
sizes $x$ and $y$, respectively. Therefore we have that
$$\eta_G(x,y)\geq \max\{x(\alpha n+y-n),y(\alpha n+x-n)\}.$$

Assume first that $x\leq y$ (and therefore, the maximum in the right
hand side of the above inequality is $x(\alpha n+y-n)$). Since
$\eta_G(x,y)\geq x(\alpha n+y-n)$, it follows that for each such $X$
and $Y$ we have that $e_G(X,Y)\geq x(\alpha n+y-n)$. Applying
Chernoff and the union bound, using the fact that $\alpha
n+y-n=\Theta(n)$ (here we use that $\alpha-1/2\geq c>0$ for some
constant $c$) we obtain that
$$\Pr\left(\exists \text{ such }X,Y \text{ with } e_{G_p}(X,Y)\leq
(1-\varepsilon)e_G(X,Y)p\right)\leq
4^n\cdot\sum_{x=f(n)}^{n}e^{-\Theta(xnp)}$$
$$=e^{-\omega(n)}.$$

By symmetry, the above estimate is valid for $y\leq x$ as well.

Now, recall that $G$ contains a $\rho n$-factor and hence by
Proposition \ref{Pr1} satisfies $\rho nx\leq e_G(X,Y)+\rho n(n-y)$
for all $X\subseteq A$ and $Y\subseteq B$. Multiply both sides by
$(1-\varepsilon)p$. Since if $X$ and $Y$ satisfy the assumptions of
Case 3, we have with probability $1-n^{-\omega(1)}$ that
$e_{G_p}(X,Y)\geq(1-\varepsilon)e_{G}(X,Y)p$, and it follows that
$$(1-\varepsilon)\rho np|X|\leq (1-\varepsilon)e_{G}(X,Y)p+(1-\varepsilon)\rho np(n-|Y|)$$
$$\leq e_{G_p}(X,Y)+(1-\varepsilon)\rho np(n-|Y|)$$
holds for each $X\subseteq A$ and $Y\subseteq B$ covered by Case 3.
Therefore, by Proposition \ref{Pr1} we conclude that with
probability $1-n^{-\omega(1)}$ the random subgraph $G_p$ contains a
$(1-\varepsilon)\rho np$-factor as desired.
\end{proof}

\subsection{Properties of random partitions of
vertices}\label{subsec:properties} In this subsection we introduce
several lemmas about properties of random partitions of vertices of
dense hypergraphs. The following lemma shows that the vertex set of
a dense $k$-uniform hypergraph can be partitioned in such a way that
the proportion of the degrees to each part remains about the same as
in the hypergraph.

\begin{lemma}\label{key1}
Let $k$ be a positive integer and let $\delta>0$ and $\varepsilon>0$
be real numbers. Then, for every $c>0$ and a sufficiently large
integer $n$, the following holds. Suppose that
\begin{enumerate}[$(i)$]
\item $\mathcal H$ is a $k$-uniform hypergraph with
$n$ vertices, and
\item $\delta_{k-1}(\mathcal H)\geq \delta n+\varepsilon n$, and
\item $m_1,\ldots, m_{t}$ are integers such that $m_i\geq c n$ for $1\leq i\leq t$, and
$m_1+\ldots+m_t=n$, and
\item $V(\mathcal H)=V_1\cup \ldots \cup V_{t}$ is a partition of $V(\mathcal H)$, chosen uniformly
at random among all partitions into $t$ parts, with part $V_i$ of
size exactly $m_i$ for every $1\leq i \leq t$.
\end{enumerate}
Then, with probability $1-e^{-\Theta(n)}$ the following holds:
\[\delta_{k-1}(V_i)\geq (\delta+2\varepsilon/3)m_i \textrm{ for
every } 1\leq i \leq t.\]
\end{lemma}

\begin{proof}
Let $V(\mathcal H)=V_1\cup\ldots \cup V_{t}$ be a random partition
of $V(\mathcal H)$ into $t$ parts, each of size exactly $m_i$, and
set
$$a_i=(\delta+2\varepsilon/3)m_i.$$ Now, note that for each $X\in
\binom{V(\mathcal H)}{k-1}$ and for each $1\leq i \leq t$, the
parameter $d_{\mathcal H}(X,V_i)$ has a hypergeometric distribution
with mean $\mu\geq (\delta+\varepsilon)m_i$. Therefore, by Lemma
\ref{Che} it follows that
$$\Pr\left[d_{\mathcal H}(X,V_i)< a_i \right]\leq e^{-\Theta(m_i)}=e^{-\Theta(n)}.$$

Applying the union bound we obtain that
$$\Pr\left[\exists X\in \binom{V(\mathcal H)}{k-1} \textrm{ and } 1\leq i\leq
t \textrm{ such that } d_{\mathcal H}(X,V_i)<a_i\right]\leq
\Theta(n^{k-1})e^{-\Theta(n)}=e^{-\Theta(n)}.$$

This completes the proof. \end{proof}

Let $\mathcal H$ be a $k$-uniform hypergraph on $n$ vertices and let
$0\leq \ell \leq k/2$ be an integer. Assume in addition that $n$ is
divisible by $k-\ell$ and that our goal is to find Hamilton
$\ell$-cycles in $\mathcal H$. We distinguish between two cases and
for each case, in a similar way as in \cite{FK}, we define an
auxiliary graph that will serve us throughout the paper.

(1) Case $1\leq \ell<k/2$. Let $V(\mathcal H)=A\cup B$ be a
partition of $V(\mathcal H)$ for which $|A|=\ell\cdot
\frac{n}{k-\ell}$. Let $\mathcal M_A=(F_0,\ldots, F_{m-1})$ be a
sequence of $m:=\frac{n}{k-\ell}$ disjoint $\ell$-subsets of $A$ and
let $\mathcal M_B$ be a (non-ordered) collection of
$\frac{|B|}{k-2\ell}=(n-\ell\cdot \frac{n}{k-\ell})/(k-2\ell)=m$
disjoint $(k-2\ell)$-subsets of $B$. Note that $\mathcal M_A$ can be
considered as a spanning $(2\ell,\ell)$-cycle of $A$ and $\mathcal
M_B$ as a perfect matching of the complete $(k-2\ell)$-uniform
hypergraph on the vertex set $B$. Define an auxiliary bipartite
graph $G_{\mathcal H}:=G(\mathcal M_A,\mathcal M_B,\mathcal
H)=(S\cup T,E)$, with both parts of size $|S|=|T|=m$, as follows:
\begin{enumerate} [$(i)$]
\item $S:=\left\{ F_{i}F_{i+1}: 0\leq i\leq m-1
\right\}$ (we refer to $m$ as $0$), and
\item $T:=\mathcal M_B$, and
\item for $s\in S$ and $t \in T$, $st\in E$ if and only if $t\cup
F_{i}\cup F_{i+1}\in E(\mathcal H)$, where $i$ is the unique integer
for which $s=F_{i}F_{i+1}$.
\end{enumerate}

A moment's thought now reveals that there is an injection between
perfect matchings of $G_{\mathcal H}$ and Hamilton $\ell$-cycles of
$\mathcal H$. This fact is used extensively throughout the paper.

(2) Case $\ell=0$ (note that a Hamilton $0$-cycle is a perfect
matching). Here we take a partition $V(\mathcal H)=A\cup B$ into two
sets $A$ and $B$ such that $|A|=\frac{\lfloor k/2\rfloor\cdot
n}{k}$. Let $\mathcal M_A$ be a collection of $\frac{n}{k}$ disjoint
subsets of $A$, each of size exactly $\lfloor k/2\rfloor$, and let
$\mathcal M_B$ be a collection of $\frac{n}{k}$ disjoint subsets of
$B$, each of size exactly $\lceil k/2\rceil$. Define an auxiliary
bipartite graph $G_{\mathcal H}:=G(\mathcal M_A,\mathcal
M_B,\mathcal H)=(S\cup T,E)$, with parts $S$ and $T$ as follows:
\begin{enumerate} [$(i)$]
\item $S=\mathcal M_A$ and $T=\mathcal M_B$, and
\item for $s\in S$ and $t \in T$, $st\in E$ if and only if $s\cup
t\in E(\mathcal H)$.
\end{enumerate}

Note that in this case every perfect matching in $G_{\mathcal H}$
corresponds to a perfect matching (a Hamilton $(k,0)$-cycle) of
$\mathcal H$.

%

The following lemma shows that by picking $V(\mathcal H)=A\cup B$,
$\mathcal M_A$ and $\mathcal M_B$ at random, the auxiliary graph
$G_{\mathcal H}$ typically possesses some desirable properties.

\begin{lemma} \label{key2}
Let $\ell$ and $k$ be integers for which $0\leq \ell<k/2$. Let
$\delta>0$ and $\varepsilon>0$ be real numbers. Then, for
sufficiently large integers $n$ the following holds. Suppose that
\begin{enumerate} [$(i)$]
\item $(k-\ell)|n$, and
\item $\mathcal H$ is a $k$-uniform hypergraph on $n$
vertices, and
\item $\delta_{k-1}(\mathcal H)\geq \delta n+\varepsilon n$.
\end{enumerate}

Then, for a random uniform choice of $A$, $B$, $\mathcal M_A$ and
$\mathcal M_B$ as described above, with probability
$1-e^{-\Theta(n)}$ we get that $\delta(G_{\mathcal H})\geq
(\delta+\varepsilon/2)m$, where $m=|\mathcal M_A|$.
\end{lemma}

\begin{proof}
First, consider the case where $1\leq\ell<k/2$. Let $V(\mathcal
H)=A\cup B$ be a typical partition as obtained by Lemma \ref{key1}
with $m_1=\frac{\ell}{k-\ell}\cdot n$ and $m_2=n-m_1$. The
conclusion of Lemma \ref{key2} for this case is an immediate
consequence of the following two claims:

\begin{claim} \label{1}
With probability $1-e^{-\Theta(n)}$ a random collection $\mathcal
M_B$ as described above is such that
$$|\left\{Y\in \mathcal M_B: X\cup Y\in E(\mathcal H)\right\}|\geq
(\delta+\varepsilon/2)|\mathcal M_B|$$ holds for each $X\in
\binom{V(\mathcal H)}{2\ell}$. In particular, $d_{G_{\mathcal
H}}(s)\geq (\delta+\varepsilon/2)m$ for every $s\in S$.
\end{claim}

\begin{proof} We pick $\mathcal M_B$ as follows: Let $\{v_0,\ldots,v_{|B|-1}\}$ be a random enumeration of
the elements of $B$ and define $$\mathcal
M_B:=\left\{\{v_{j},\ldots, v_{j+k-2\ell-1}\}: j=(k-2\ell)i, 0\leq
i\leq m-1\right\}.$$ Now, for a subset $X\in \binom{V(\mathcal
H)}{2\ell}$, define $d_B(X)=|\left\{Y\in \mathcal M_B: X\cup Y\in
E(\mathcal H)\right\}|$. We wish to show that
$$\Pr\left[\exists X\in \binom{V(\mathcal H)}{2\ell} \textrm{ such
that } d_B(X)<(\delta+\varepsilon/2)m\right]=e^{-\Theta(n)}.$$

Indeed, fix $X\in \binom{V(\mathcal H)}{2\ell}$, and for each $0\leq
i\leq m-1$, let $Y_i$ be the indicator random variable for the event
``$X\cup \{v_j,\ldots, v_{j+k-2\ell-1}\}\in E(\mathcal H)$", where
$j=(k-2\ell)\cdot i$. Starting the enumeration of the elements of
$B$ from the $j^{th}$ place and using the fact that $$d_{\mathcal
H}\left(X\cup\{v_j,\ldots,v_{j+k-2\ell-2}\},B\right)\geq
(\delta+2\varepsilon/3)|B|,$$ we obtain that $\mathbb{E}(Y_i)\geq
\delta+2\varepsilon/3$, for every $0\leq i\leq m-1$. Hence,
$$\mathbb{E}(d_B(X))=\sum_{i=0}^{m-1} \mathbb{E}(Y_i)\geq
(\delta+2\varepsilon/3)m.$$

Now, given an enumeration of $B=\{v_0,\ldots,v_{|B|-1}\}$, by
switching between two elements $v_i$ and $v_j$, $d_B(X)$ can change
by at most $2$. Therefore, using Theorem \ref{talagrand} it follows
that

\begin{eqnarray*}
\Pr\left[d_B(X)<(\delta+\varepsilon/2)|\mathcal M_B|\right] & \leq &
\Pr\left[d_B(X)<\mathbb{E}(d_B(X))-\varepsilon m/6 \right] \\
&\leq &2\exp\left(\frac{-\varepsilon^2m^2}{|B|72}\right)=e^{-\Theta(n)}. \nonumber\\
\end{eqnarray*}

Applying the union bound we obtain that

$$\Pr\left[\exists X\in
\binom{V(\mathcal H)}{2\ell} \textrm{ such that }
d_B(X)<(\delta+\varepsilon/2)|\mathcal M_B|\right]\leq
\Theta(n^{2\ell})e^{-\Theta(n)}=e^{-\Theta(n)}$$

as desired.
\end{proof}

\begin{claim} \label{2}
With probability $1-e^{-\Theta(n)}$, a random (enumerated)
collection $\mathcal M_A=\{F_0,\ldots,F_{m-1}\}$ as described above
is such that
$$|\left\{i: 0\leq i\leq m-1 \textrm{ and } X\cup F_i\cup F_{i+1}\in E(\mathcal H)\right\}|\geq
(\delta+\varepsilon/2)m$$ holds for each $X\in \binom{V(\mathcal
H)}{k-2\ell}$. In particular, $d_{G_{\mathcal H}}(t)\geq
(\delta+\varepsilon/2)m$ for every $t\in T$.
\end{claim}

\begin{proof} We pick $\mathcal M_A$ as follows: Let
$\{u_0,\ldots,u_{|A|-1}\}$ be a random enumeration of the elements
of $A$, and for each $0\leq i\leq m-1$, define $$F_i=\{u_{\ell\cdot
i},\ldots, u_{\ell\cdot(i+1)-1}\}$$ and set
$$\mathcal M_A=\{F_0,\ldots,F_{m-1}\}.$$

Now, for a subset $X\in \binom{V(\mathcal H)}{k-2\ell}$, define
$d_A(X)=|\left\{i: 0\leq i\leq m-1 \textrm{ and } X\cup F_i\cup
F_{i+1}\in E(\mathcal H)\right\}|$, we wish to show that
$$\Pr\left[\exists X \in \binom{V(\mathcal H)}{k-2\ell} \textrm{ such that }
d_A(X)<(\delta+\varepsilon/2)m\right]=e^{-\Theta(n)}.$$

Indeed, fix $X\in \binom{V(\mathcal H)}{k-2\ell}$, and for each
$0\leq i\leq m-1$, let $X_i$ be the indicator random variable for
the event ``$X\cup F_i\cup F_{i+1}\in E(\mathcal H)$". From here,
the proof is similar to the proof of Claim \ref{1} so we omit the
details (the only difference is that here, switching two elements
can change $d_A(X)$ by at most $4$ and not $2$, which does not cause
any problem).
\end{proof}

Next, consider the case where $\ell=0$. In this case, let
$V(\mathcal H)=A\cup B$ be a typical partition as obtained by Lemma
\ref{key1} with $t=2$ and $m_1=\frac{\lfloor k/2\rfloor n}{k}$. Now,
randomly define $\mathcal M_A$ and $\mathcal M_B$ as described
above. Finally, Claim \ref{1} shows that with high probability we
obtain $\delta(G_{\mathcal H})\geq (\delta+\varepsilon/2)m$ as
desired.


This completes the proof.
\end{proof}

\begin{remark}\label{remark1}
If we change Condition $(iii)$ of Lemma \ref{key2} to
$(\delta-\varepsilon)n\leq \delta_{k-1}(\mathcal H)\leq
\Delta_{k-1}(\mathcal H)\leq (\delta+\varepsilon)n$, then the same
proof (more or less line by line) shows that $d_{G_{\mathcal
H}}(v)\in(\delta\pm2\varepsilon)m$ for every $v\in V(G_{\mathcal
H})$. We will make use of this fact in the proof of Theorem
\ref{main3}.
\end{remark}

\section{Proofs of the main results}

\subsection{Proof of Theorem \ref{main1}}

In this subsection we prove Theorem \ref{main1}.

\begin{proof} In order to prove Theorem \ref{main1} we show that for every $\varepsilon>0$, the number of Hamilton $\ell$-cycles in $\mathcal H$ is at least
$$(1-o(1))^n\cdot n!\cdot\left(\frac{\alpha-\varepsilon/2}{\ell!(k-2\ell)!}\right)^{\frac{n}{k-\ell}}.$$

Let $\varepsilon>0$ be a positive constant. Denote
$\delta=\alpha-\varepsilon$ and observe that $\delta_{k-1}(\mathcal
H)\geq (\delta+\varepsilon)n$. First, consider the case where $1\leq
\ell<k/2$. Assume that $V(\mathcal H)=A\cup B$ is a partition of
$V(\mathcal H)$ into two sets $A$ and $B$ with
$|A|=\ell\cdot\frac{n}{k-\ell}$, equipped with $\mathcal M_A$ and
$\mathcal M_B$ as described in Subsection \ref{subsec:properties}.
By applying Lemma \ref{key2} to $\mathcal H$, it follows that a
$(1-o(1))$-fraction of these partitions are such that
$\delta(G_{\mathcal H})\geq (\delta+\varepsilon/2)m$ (where
$m=\frac{n}{k-\ell}$ and $G_{\mathcal H}$ is the auxiliary graph as
defined in Subsection \ref{subsec:properties}). Now, using Theorem
\ref{CK} we obtain that the number of perfect matchings in each such
$G_{\mathcal H}$ is at least
\begin{align*}
&(1-o(1))^n\left(\delta+\varepsilon/2\right)^m\cdot
m!=(1-o(1))^{n}(\delta+\varepsilon/2)^{\frac{n}{k-\ell}}\left(\frac{n}{k-\ell}\right)!.
\end{align*}
Next, note that each perfect matching of $G_{\mathcal H}$
corresponds to a Hamilton $\ell$-cycle. Moreover, note that given
two such partitions $A\cup  B$ and $A'\cup B'$ of $V(\mathcal H)$,
if $A\neq A'$, then clearly the obtained Hamilton $\ell$-cycles
coming from these partitions are distinct. In case $A=A'$, note that
as long as $\mathcal M_A$ and $\mathcal M_{A'}$ do not define the
same cyclic ordering for the same sets $F_0,\ldots,F_{m-1}$ (for a
fixed $\mathcal M_A$ there are at most $2m$ such $\mathcal M_A'$),
the Hamilton $\ell$-cycles obtained from such distinct structures
are all distinct. All in all, combining the above mentioned, each
Hamilton $\ell$-cycle is being counted at most $2m$ times and we
obtain that the number of Hamilton $\ell$-cycles in $\mathcal H$ is
at least

\begin{align*}
&\frac{1}{2m}\cdot(1-o(1))^n\cdot
n!\cdot\left(\frac{\delta+\varepsilon/2}{\ell!(k-2\ell)!}\right)^{\frac{n}{k-\ell}}\\
&=(1-o(1))^n\cdot
n!\cdot\left(\frac{\alpha-\varepsilon/2}{\ell!(k-2\ell)!}\right)^{\frac{n}{k-\ell}}.
\end{align*}

Indeed, we need to multiply the above estimate by the number of
auxiliary graphs $G_{{\cal H}}$. For this, take a permutation of
$V(\mathcal H)$, define $A$ to be its first
$\ell\cdot\frac{n}{k-\ell}$ vertices, $\mathcal M_A$ to be the first
$\frac{n}{k-\ell}$ consecutive (and disjoint) $\ell$-tuples, and
$\mathcal M_B$ to be the last $\frac{n}{k-\ell}$ consecutive
$(k-2\ell)$-tuples. Then, divide by the ordering inside the tuples
and the ordering between the tuples in $\mathcal M_B$.

Next, for the case where $\ell=0$ the proof is more or less the
same. Here, the partitions we consider are of the form $(A,B)$ where
$|A|=\frac{\lfloor k/2\rfloor\cdot n}{k}$, and $\mathcal M_A$ in the
definition of the auxiliary graph $G_{\mathcal H}$ is just a
collection of sets, not enumerated. In addition, every perfect
matching can be obtained by $\binom{k}{\lfloor
k/2\rfloor}^{\frac{n}{k}}$ partitions (from each edge, choose
$\lfloor k/2\rfloor$ elements to be in $A$). All in all, there are
at least

\begin{align*}
&\frac{n!}{\left(\lfloor k/2\rfloor!\lceil
k/2\rceil!\right)^{\frac{n}{k}}}
\cdot\frac{\left((1-o(1))(\delta+\varepsilon/2)\right)^{\frac{n}{k}}(n/k)!}{\binom{k}{\lfloor
k/2\rfloor}^{\frac{n}{k}}(n/k)!(n/k)!}\\
&=\frac{n!}{(k!)^{\frac{n}{k}}(n/k)!}\cdot
\left((1-o(1))(\delta+\varepsilon/2)\right)^{\frac{n}{k}}\\
&=(1-o(1))^n\cdot
n!\cdot\left(\frac{\alpha-\varepsilon/2}{k!}\right)^{\frac{n}{k}}
\end{align*}
perfect matchings in $\mathcal H$. This completes the proof of
Theorem \ref{main1}.
\end{proof}
\subsection{Proofs of Theorems \ref{main2} and \ref{main3}}
In this subsection we prove Theorems \ref{main2} and \ref{main3}. We
start with Theorem \ref{main2}.

\begin{proof} The proof is rather similar to the proof of the main
results in \cite{FK}. The main difference is that here we use
Theorem \ref{RandomKaterinis} in order to find many edge-disjoint
perfect matchings in random subgraphs of a graph which is not
necessarily the complete bipartite graph. Let
$\varepsilon=\alpha-\alpha'$, and note that $\delta_{k-1}(\mathcal
H)\geq (\alpha'+\varepsilon)n$. We distinguish between two cases:

{\bf Case I:} $1\leq \ell<k/2$. In this case the general scheme goes
as follows:

First, choose $r:=|E(\mathcal H)|\cdot\left(\frac{(k-\ell)\ln
n}{n}\right)^2$ random partitions of $V(\mathcal H)$,
$\{(A_i,B_i):1\leq i\leq r\}$, such that
$|A_i|=\ell\cdot\frac{n}{k-\ell}$ for each $i$, equipped with
$\mathcal M_{A_{i}}$ and $\mathcal M_{B_{i}}$ as described in
Section \ref{subsec:properties}. For each such partition
$(A_i,B_i)$, denote the corresponding auxiliary graph $G_{\mathcal
H}$ by $G^{(i)}$, and use the notation $\mathcal
M_{A_i}=\{F_{i,0},\ldots,F_{i,m-1}\}$, where $m=\frac{n}{k-\ell}$.
Note that by Lemma \ref{key2} we have that w.h.p
$\delta(G^{(i)})\geq (\alpha'+\varepsilon/2)m$ for every $1\leq
i\leq
r$. 

Second, for each edge $f\in E(\mathcal H)$, we say that $i$ is a
\emph{candidate} for $f$ if there exist $j$ and $B\in \mathcal
M_{B_i}$ such that $f=F_{i,j}\cup B\cup F_{i,j+1}$. For each edge
$f\in E(\mathcal H)$ let $\psi(f)$ denote the number of candidates
it has. Each $f\in E(\mathcal H)$ with $\psi(f)>0$ picks one
candidate $i$ at random among the $\psi(f)$ candidates. For each
$1\leq i\leq r$, consider the subhypergraph $\mathcal H_i$ obtained
from the partition $(A_i,B_i)$ together with the edges that chose
$i$, and denote the corresponding auxiliary subgraph of $G^{(i)}$ by
$H_i$. Observe that the hypergraphs ${\cal H}_i$ are edge-disjoint.

Finally, we wish to show that w.h.p every auxiliary graph $H_i$
contains $(1-o(1))\frac{f(\alpha')m}{\ln^2n}$ edge-disjoint perfect
matchings, where $f(\alpha')=\frac{\alpha'+\sqrt{2\alpha'-1}}{2}$.
We then conclude that every subhypergraph $\mathcal H_i$ contains
$(1-o(1))\frac{f(\alpha')m}{\ln^2n}$ edge-disjoint Hamilton
$\ell$-cycles for each $i$, and therefore $\mathcal H$ contains at
least
$$(1-o(1))r\cdot\frac{f(\alpha')m}{\ln^2n}=(1-o(1))\frac{|E(\mathcal
H)|\cdot f(\alpha')}{\frac{n}{k-\ell}}$$ edge-disjoint Hamilton
$\ell$-cycles as required. To this end we need the following claim:

\begin{claim}\label{claim3}
With high probability the following holds: every $H_i$ contains at
least $(1-o(1))\frac{f(\alpha')n}{\ln^2n}$ edge-disjoint perfect
matchings.
\end{claim}

\begin{proof} Let $f\in E(\mathcal H)$ be an edge and recall that
the random variable $\psi(f)$ counts the number of partitions
$(A_i,B_i)$ which are candidates for $f$. Observe that for every edge $f$ and index $i$, the $i^{th}$
partition is a candidate for $f$ with the same probability

$$q\leq\frac{|\mathcal M_{A_i}|\cdot|\mathcal M_{B_i}|}{|E(\mathcal H)|}=\frac{m^2}{|E(\mathcal H)|}.$$

(Indeed, the partition $(A_i,B_i)$ is a candidate for at most
$|\mathcal M_{A_i}|\cdot|\mathcal M_{B_i}|=m^2$ edges, among all the
$|E(\mathcal H)|=\Theta(n^k)$ edges of the hypergraph. Due to
symmetry this bound is obtained.)

Therefore, since $\psi(f)\sim \Bin(r,q)$, applying Chernoff  and the
union bound we obtain that w.h.p $\psi(f)\leq (1+o(1))r\cdot q\leq
(1+o(1))\ln^2n$ for each $f\in E(\mathcal H)$. Hence, we conclude
that the edges of $G^{(i)}$ remain in $H_i$ with probability $p\geq
\frac{1-o(1)}{\ln^2n}$. Now, combining Theorem \ref{Csaba}, Remark
\ref{remark:notsameprobability}, and applying the union bound we
conclude that w.h.p. every $H_i$ contains a
$(1-o(1))\frac{f(\alpha')m}{\ln^2n}$-factor. In order to complete
the proof, recall that $H_i$ is bipartite and thus each such factor
can be decomposed to $(1-o(1))\frac{f(\alpha')m}{\ln^2n}$
edge-disjoint perfect matchings.
\end{proof}

{\bf Case II:} $\ell=0$. The proof for this case is similar to
previous case so we omit it. The only difference is that here we use
a slightly different auxiliary graph, so for this case we need to
take $r=|E(\mathcal H)|\cdot \left(\frac{k\ln n}{n}\right)^2$
partitions $(A_i,B_i)$ with $|A_i|=\frac{\lfloor k/2\rfloor\cdot
n}{k}$. All the other calculations remain the same.

This completes the proof of Theorem \ref{main2}.
\end{proof}

Now we prove Theorem \ref{main3}.

\begin{proof} The proof of Theorem \ref{main3} is quite similar to the
previous proof, so we might omit few details. Let $\delta>0$ be a
constant, let $\varepsilon>0$ be a sufficiently small constant (to
be determined later), and let $\mathcal H$ be a $k$-uniform
hypergraph which satisfies the assumptions of the theorem.
Throughout the proof we use similar notation as in the proof of
Theorem \ref{main2}.

First, let $(A,B)$ be a random partition of $V(\mathcal H)$ into two
sets with $|A|=\ell\cdot\frac{n}{k-\ell}$, equipped with $\mathcal
M_A$ and $\mathcal M_B$ as described in Section
\ref{subsec:properties}. Using Remark \ref{remark1} we conclude that
with probability $1-e^{-\Theta(n)}$ we have
$(\alpha-\varepsilon)m\leq \delta(G_{\mathcal H})\leq
\Delta(G_{\mathcal H}) \leq (\alpha+2\varepsilon)m$. Conditioning on
that, similarly to the calculation in Claim \ref{claim3}, we
conclude that for such a partition $(A,B)$ and an edge $f\in
E(\mathcal H)$, the probability that $(A,B)$ is a candidate for $f$
is bounded between $$\frac{|\mathcal
M_{A}|\cdot(\alpha-\varepsilon)m}{|E(\mathcal
H)|}=\frac{(\alpha-\varepsilon)m^2}{|E(\mathcal H)|}$$ and
$$\frac{|\mathcal M_{A}|\cdot(\alpha+2\varepsilon)m}{|E(\mathcal
H)|}=\frac{(\alpha+2\varepsilon)m^2}{|E(\mathcal H)|}.$$

Second, let $q=\frac{(\alpha-\varepsilon)m^2}{|E(\mathcal H)|}$
(clearly, $q$ is a lower bound for that probability), and choose
$r:= |E(\mathcal H)|\cdot\left(\frac{k-\ell}{n}\right)^2\cdot \frac
1q$ random partitions of $V(\mathcal H)$, $\{(A_i,B_i):1\leq i\leq
r\}$, such that $|A_i|=\ell\cdot\frac{n}{k-\ell}$ for each $i$,
equipped with $\mathcal M_{A_{i}}$ and $\mathcal M_{B_{i}}$ as
described in Section \ref{subsec:properties} (and in the proof of
Theorem \ref{main2}).

Third, since $\psi(f)$ is binomially distributed with probability
$q\leq q_f\leq (1+7\varepsilon)q$, by Chernoff's inequality and the
union bound we obtain that $\psi(f)\in (1\pm 8\varepsilon)rq$ holds
for each $f\in E(\mathcal H)$.

Next, using the fact that all the $G^{(i)}$'s are almost regular
(all the degrees lie in the interval $(\alpha\pm 2\varepsilon)m$),
combining Theorem \ref{FactorInAlmostRegular} with Theorem
\ref{RandomKaterinis}, using the fact that $\varepsilon$ is
sufficiently small, we obtain that with probability
$1-n^{-\omega(1)}$ each $H_i$ contains at least
$(1-o(1))(\alpha-20\sqrt{2\varepsilon})\frac{m}{rq}$ edge-disjoint
perfect matchings. Therefore, for each $i$, by taking all the
edge-disjoint Hamilton $\ell$-cycles in $\mathcal H_i$, there is at
most a $40\sqrt{2\varepsilon}$-fraction of edges in $\mathcal H_i$
which are unused. All in all, there is at most
$40\sqrt{2\varepsilon}$-fraction of edges in $\mathcal H$ which are
not covered by any of the Hamilton $\ell$-cycles. Finally, by taking
$\varepsilon$ to be small enough such that
$40\sqrt{2\varepsilon}\leq \delta$ we complete the proof.
\end{proof}

%
%
%

\subsection{Proof of Proposition \ref{main4}}

\begin{proof} Let $k\leq n$ be positive integers. Define a
$k$-uniform hypergraph $\mathcal H$ on $n$ vertices as follows: Let
$V(\mathcal H)=[n]$, and partition $V(\mathcal H)=A\cup B$ into two
sets $A$ and $B$ such that $n/2-1\leq |A|\leq n/2+1$ is an odd
integer. Let $E(\mathcal H)$ consists of all the $k$-tuples
$f\in\binom{[n]}{k}$ for which $|A\cap f|$ is even, and observe that
$\delta_{k-1}(\mathcal H)\geq n/2-k$. Now, let $r$ be an odd integer
and assume towards a contradiction that $\mathcal H$ contains an
$r$-factor $\mathcal H'\subseteq \mathcal H$. Let $\mathcal H''$ be
the multi-hypergraph on the vertex set $A$ which consists of the
(multi-)set of edges $\{A\cap f:f\in E(\mathcal H')\}$. Since all
the edges of $\mathcal H''$ are of even size, the size of $A$ is
odd, and since all the vertex degrees are $r$ (which is odd), we
derive a contradiction.
\end{proof}

\section{Concluding remarks and open problems}

To the best of out knowledge, this paper is the first to deal with
problems of counting and packing in general dense hypergraphs. Here
we have obtained some preliminary results, which suggest many
interesting and challenging problems for further study.

In Theorem \ref{main1} we showed that, for every $\ell<k/2$, the
number of Hamilton $\ell$-cycles in $k$-uniform hypergraphs with
large minimum degree is lower bounded (up to sub-exponential factor)
with the expected number of such cycles in a random hypergraph with
the same density. It would be interesting to generalize it to every
$\ell<k$.

In Theorems \ref{main2} and \ref{main3} we dealt with the question
of packing Hamilton $\ell$-cycles into dense $k$-uniform hypergraph.
We showed that for $\ell<k/2$, if $\delta_{k-1}(\mathcal H)\geq
\alpha n$, for some $\alpha>1/2$, then one can find
$\frac{f(\alpha)|E(\mathcal H)|}{\frac{n}{k-\ell}}$ edge-disjoint
Hamilton $\ell$-cycles. It is natural to try to obtain the best
possible $f(\alpha)$, and to try to generalize our results for every
$\ell\leq k-1$.

As was mentioned in the introduction, K\"uhn, Mycroft and Osthus
showed in \cite{KMO} that $\delta_{k-1}\approx \frac{n}{\lceil
\frac{k}{k-\ell}\rceil(k-\ell)}$ is the correct asymptotic bound for
the existence of a Hamilton $\ell$-cycle. Note that for certain
choices of $k$ and $\ell$ (for example, $k=3$ and $\ell=1$), this
bound is much smaller than the bound of $n/2$ that we considered. It
would be nice to extend our results to hypergraphs with minimum
degrees starting at $\frac{n}{\lceil
\frac{k}{k-\ell}\rceil(k-\ell)}$, for every $\ell< k$.

{\bf Acknowledgement.} A major part of this work was carried out
when Benny Sudakov was visiting Tel Aviv University, Israel. He
would like to thank the School of Mathematical Sciences of Tel Aviv
University for hospitality and for creating a stimulating research
environment.

\end{document}